\documentclass[12pt]{amsart}
\usepackage{amsmath}
\usepackage{amssymb}
\usepackage[all]{xy}
\usepackage{longtable}
\usepackage[mathscr]{eucal}
\usepackage{xcolor}
\usepackage[bookmarks=false]{hyperref}
\usepackage{multirow,bigdelim}
\usepackage{nccmath}
\usepackage{url}
\usepackage{ulem}
\usepackage{comment}
\usepackage{mathrsfs}
\usepackage{arydshln}
\usepackage{mathtools}

\allowdisplaybreaks[3]

\numberwithin{equation}{subsection}

\newtheorem{theorem}[equation]{Theorem}
\newtheorem{proposition}[equation]{Proposition}
\newtheorem{lemma}[equation]{Lemma}
\newtheorem{corollary}[equation]{Corollary}
\newtheorem{conjecture}[equation]{Conjecture}
\newtheorem{problem}[equation]{Problem}

\theoremstyle{definition}
\newtheorem{example}[equation]{Example}
\newtheorem{definition}[equation]{Definition}

\theoremstyle{remark}
\newtheorem{remark}[equation]{Remark}

\newcommand{\ZZ}{\mathbb{Z}}
\newcommand{\QQ}{\mathbb{Q}}
\newcommand{\RR}{\mathbb{R}}

\newcommand{\CC}{\mathbb{C}}

\newcommand{\Fq}{\mathbb{F}_{q}}
\newcommand{\Fp}{\mathbb{F}_{p}}

\newcommand{\cH}{\mathcal{H}}
\newcommand{\cI}{\mathcal{I}}

\newcommand{\cL}{\mathcal{L}}

\newcommand{\cZ}{\mathcal{Z}}

\newcommand{\ocL}{\overline{\mathcal{L}}}
\newcommand{\ocZ}{\overline{\mathcal{Z}}}

\newcommand{\IT}{\mathcal{I}^{\mathrm{T}}}
\newcommand{\ITw}{\mathcal{I}^{\mathrm{T}}_{w}}

\newcommand{\sI}{\mathscr{I}}
\newcommand{\sL}{\mathscr{L}}
\newcommand{\sLL}{\mathscr{L}^{\Li}}
\newcommand{\sLz}{\mathscr{L}^{\zeta}}
\newcommand{\sLb}{\mathscr{L}^{\bullet}}
\newcommand{\sLLinv}{\mathscr{L}^{\Li, *\textnormal{-inv}}}
\newcommand{\sLzinv}{\mathscr{L}^{\zeta, *\textnormal{-inv}}}
\newcommand{\sLbinv}{\mathscr{L}^{\bullet, *\textnormal{-inv}}}

\DeclareMathAlphabet{\matheur}{U}{eur}{m}{n}

\newcommand{\fZ}{\mathfrak{Z}}
\newcommand{\fa}{\mathfrak{a}}

\newcommand{\fc}{\mathfrak{c}}
\newcommand{\fs}{\mathfrak{s}}
\newcommand{\fm}{\mathfrak{m}}

\newcommand{\fh}{\mathfrak{h}}

\newcommand{\fn}{\mathfrak{n}}

\newcommand{\sA}{\mathscr{A}}
\newcommand{\sB}{\mathscr{B}}
\newcommand{\sC}{\mathscr{C}}
\newcommand{\sR}{\mathscr{R}}

\newcommand{\sU}{\mathscr{U}}

\newcommand{\sBC}{\mathscr{BC}}

\DeclareMathOperator{\Ker}{Ker}

 \DeclareMathOperator{\wt}{wt}

\DeclareMathOperator{\Li}{Li}

\DeclareMathOperator{\dep}{dep}
\DeclareMathOperator{\Span}{Span}

\DeclareMathOperator{\Supp}{Supp}

\DeclareFontEncoding{OT2}{}{}
\DeclareFontSubstitution{OT2}{cmr}{m}{n}

\DeclareFontFamily{OT2}{cmr}{\hyphenchar\font45}
\DeclareFontShape{OT2}{cmr}{m}{n}{%
   <5><6><7><8><9>gen*wncyr%
   <10><10.95><12><14.4><17.28><20.74><24.88>wncyr10}{}
\DeclareFontShape{OT2}{cmr}{b}{n}{%
   <5><6><7><8><9>gen*wncyb%
   <10><10.95><12><14.4><17.28><20.74><24.88>wncyb10}{}

\DeclareMathAlphabet{\mathcyr}{OT2}{cmr}{m}{n}
\DeclareMathAlphabet{\mathcyb}{OT2}{cmr}{b}{n}
\SetMathAlphabet{\mathcyr}{bold}{OT2}{cmr}{b}{n}

\newcommand{\sh}{\mathcyr{sh}}

\newcommand{\ok}{\overline{k}}

\newcommand{\adm}{\mathrm{adm}}

\newcommand{\tpi}{\widetilde{\pi}}

\newcommand{\zetainv}{\zeta^{*\textnormal{-inv}}}
\newcommand{\Liinv}{\Li^{*\textnormal{-inv}}}
\newcommand{\cLinv}{\mathcal{L}^{*\textnormal{-inv}}}
\newcommand{\ocLinv}{\overline{\mathcal{L}}^{*\textnormal{-inv}}}

\newcommand{\Lis}{\Li^{\star}}

\title[Involution on a quotient space of MZVs in positive characteristic]{Involution on a quotient space of multiple zeta values in positive characteristic}

\author{Yoshinori Mishiba}
\address{Mathematical Institute, Tohoku University, 6-3, Aramaki Aza-Aoba, Aoba-ku, Sendai 980-8578, Japan}
\email{mishiba@tohoku.ac.jp}

\thanks{This work was supported by JSPS KAKENHI Grant Numbers JP23K03073,JP24H00015.}



\date{\today}

\begin{document}

\begin{abstract}
In this paper, we introduce $*$-inverse multiple zeta values and $*$-inverse Carlitz multiple polylogarithms in positive characteristic and study their algebraic structures and relations. Using special values of Carlitz multiple polylogarithms, we show that a natural quotient of the space of multiple zeta values in positive characteristic admits a Hopf algebra structure and that the original space admits a compatible comodule structure over this quotient. These structures may be regarded as function field analogues of the Hopf algebra and comodule structures arising from motivic multiple zeta values in characteristic zero. In particular, the antipode on this quotient gives a non-trivial involution corresponding to the $*$-inverse operation. We also establish, in certain cases, the $q$-shuffle product formula and linear relations for $*$-inverse multiple zeta values, providing evidence for our conjectures on their relations. Finally, we consider another characteristic-zero analogue of our construction, based on the Hopf algebra structure given by the harmonic product and the deconcatenation coproduct, and formulate a related problem.
\end{abstract}

\maketitle

\section{Introduction}

\subsection{Motivation} \label{subsection-motivation}

Let $\cI \coloneq \bigsqcup_{r \geq 0} \ZZ_{\geq 1}^{r}$ be the set of indices.
The weight (resp.\ depth) of $\fs = (s_{1}, \ldots, s_{r}) \in \cI$ is defined by $\wt(\fs) \coloneq \sum_{i = 1}^{r} s_{i}$ (resp.\ $\dep(\fs) \coloneq r$).
The unique index of depth (and weight) zero is denoted by $\emptyset$.
An index $\fs = (s_{1}, \ldots, s_{r}) \in \cI$ is called admissible if $\fs \neq \emptyset$ and $s_{1} > 1$, or $\fs = \emptyset$.
Let $\cI^{\adm}$ be the set of admissible indices.
For each $\fs = (s_{1},\ldots, s_{r}) \in \cI^{\adm}$, the multiple zeta value (MZV) is defined by
\begin{align*}
    \zeta(\fs) \coloneq \sum_{m_{1} > \cdots > m_{r} \geq 1} \dfrac{1}{m_{1}^{s_{1}} \cdots m_{r}^{s_{r}}} \in \RR.
\end{align*}
Let $\fZ_{\QQ}$ be the $\QQ$-vector space spanned by all MZVs.
By the harmonic (or shuffle) product, $\fZ_{\QQ}$ forms a $\QQ$-algebra.

We recall the duality of multiple zeta values.
For each $s \in \ZZ_{\geq 1}$ and each $m \in \ZZ_{\geq 0}$, we write $\{ s \}^{m} \in \cI$ for the $m$-tuple $(s, \ldots, s)$, where $s$ is repeated $m$ times.
The concatenation of indices $\fs_{1}, \ldots, \fs_{m}$ is denoted by $(\fs_{1}, \ldots, \fs_{m})$.
Each $\fs \in \cI^{\adm}$ can be written as
\begin{align*}
    \fs = (a_{1} + 1, \{ 1 \}^{b_{1} - 1}, a_{2} + 1, \{ 1 \}^{b_{2} - 1}, \ldots, a_{n} + 1, \{ 1 \}^{b_{n} - 1}),
\end{align*}
where $a_{i}, b_{i} \in \ZZ_{\geq 1}$.
Then its dual is defined by
\begin{align*}
    \fs^{\dagger} \coloneq (b_{n} + 1, \{ 1 \}^{a_{n} - 1}, \ldots, b_{2} + 1, \{ 1 \}^{a_{2} - 1}, b_{1} + 1, \{ 1 \}^{a_{1} - 1}) \in \cI^{\adm}.
\end{align*}
The duality of MZVs is given by
\begin{align*}
    \zeta(\fs) = \zeta(\fs^{\dagger})
\end{align*}
for all $\fs \in \cI^{\adm}$ (see \cite[Section~9]{Za94}).
Therefore, the non-trivial involution $\fs \mapsto \fs^{\dagger}$ on the set $\cI^{\adm}$ induces the identity map $\zeta(\fs) \mapsto \zeta(\fs^{\dagger})$ on $\fZ_{\QQ}$.
This implies the following problem:

\begin{problem} \label{problem-motivation}
Construct a non-trivial $\QQ$-algebra involution on $\fZ_{\QQ}$, or on a subquotient of $\fZ_{\QQ}$.
\end{problem}

\begin{example}
Euler showed that $\zeta(2 n) \in \QQ^{\times} \cdot (2 \pi \sqrt{- 1})^{2 n}$ for all $n \in \ZZ_{\geq 1}$.
On the other hand, Lindemann proved that $\pi$ is a transcendental number.
Therefore, we have a non-trivial $\QQ$-algebra involution on $\QQ[\zeta(2)] = \QQ[\pi^{2}] \subset \fZ_{\QQ}$ given by
\begin{align*}
    \zeta(2) \mapsto - \zeta(2).
\end{align*}
\end{example}

Another approach to constructing involutions is provided by Hopf algebra structures.
Goncharov \cite[Section~2]{G05} introduced the Hopf algebra $\sI_{\bullet}(\{ 0, 1 \})$ over $\QQ$ of formal iterated integrals, which is connected, graded, and involutive.
Its coproduct is called the Goncharov coproduct.
Let $\fh^{1}_{\QQ} \coloneq \QQ \langle \ZZ_{\geq 1} \rangle$ be the free associative $\QQ$-algebra on $\ZZ_{\geq 1}$.
Then $\fh^{1}_{\QQ}$ forms a commutative $\QQ$-algebra with respect to the shuffle product $\sh$.
Bachmann and Tasaka \cite[Proposition~3.5]{BT18} showed that $(\fh^{1}_{\QQ}, \sh)$ is naturally isomorphic, as a $\QQ$-algebra, to a quotient Hopf algebra of $\sI_{\bullet}(\{ 0, 1 \})$.
Moreover, let $\cH$ be the $\QQ$-algebra generated by motivic multiple zeta values $\zeta^{\fm}(\fs)$ ($\fs \in \cI$) introduced by Brown \cite[Definition~2.1]{B12}.
Then we have the surjective homomorphisms
\begin{align*}
    \sI_{\bullet}(\{ 0, 1 \}) \twoheadrightarrow \fh^{1}_{\QQ} \twoheadrightarrow \cH \twoheadrightarrow \cH / \zeta^{\fm}(2) \cH
\end{align*}
of $\QQ$-algebras.
Then the Goncharov coproduct on $\sI_{\bullet}(\{ 0, 1 \})$ induces connected graded involutive Hopf algebra structures on $\fh^{1}_{\QQ}$ and $\cH / \zeta^{\fm}(2) \cH$, and a graded comodule structure on $\cH$ over $\cH / \zeta^{\fm}(2) \cH$.
In particular, the antipode on $\sI_{\bullet}(\{ 0,1 \})$ induces involutions on $\fh^{1}_{\QQ}$ and $\cH / \zeta^{\fm}(2) \cH$.
Moreover, these involutions are non-trivial, since the latter sends the nonzero element $\zeta^{\fm}(3) \bmod \zeta^{\fm}(2) \cH$ to its negative.
Since the period map
\begin{align*}
    \cH \twoheadrightarrow \fZ_{\QQ}
\end{align*}
is conjectured to be an isomorphism, we conjecture that the quotient space $\fZ_{\QQ} / \zeta(2) \fZ_{\QQ}$ admits a connected graded involutive Hopf algebra structure whose antipode is non-trivial.

\begin{remark} \label{remark-inversed-MZV}
Furusho \cite[Section~4.2]{F20} defined the inversed MZVs $\zeta^{\textnormal{inv}}(\fs)$ ($\fs \in \cI^{\adm}$) in terms of the coefficients of the inversed KZ-associator and studied their properties.
They can also be obtained by applying the antipode on $\fh^{1}_{\QQ}$ and then taking the image under the surjection $\fh^{1}_{\QQ} \twoheadrightarrow \fZ_{\QQ}$ given by shuffle regularization.
\end{remark}

In this paper, we study a function field analogue of Problem~\ref{problem-motivation}, following a similar Hopf algebraic approach.

\subsection{Positive characteristic} \label{section-positive-characteristic}

Let $q$ be a power of a prime number $p$.
Let $A \coloneq \Fq[\theta]$ be the polynomial ring in the variable $\theta$ over a finite field $\Fq$ with $q$ elements, $A_{+}$ the set of monic polynomials in $A$, $k \coloneq \Fq(\theta)$ the fraction field of $A$, $k_{\infty} \coloneq \Fq(\!(\theta^{- 1})\!)$ the completion of $k$ at the infinite place, $\CC_{\infty} \coloneq \widehat{\overline{k_{\infty}}}$ the completion of an algebraic closure of $k_{\infty}$, and $\ok$ the algebraic closure of $k$ in $\CC_{\infty}$.
Let $|-|_{\infty}$ be the $\infty$-adic absolute value on $\CC_{\infty}$ normalized by $|\theta|_{\infty} = q$.
For each $d \in \ZZ_{\geq 0}$, we set $L_{d} \coloneq \prod_{1 \leq i \leq d} (\theta - \theta^{q^{i}}) \in A$.
Throughout this paper, we fix an $\Fp[L_{1}]$-subalgebra $R$ of $\ok$, where $\Fp$ is the prime field of $\Fq$.

For each $s \in \ZZ_{\geq 1}$ and $d \in \ZZ_{\geq 0}$, we set
\begin{align*}
    S_{d}(s) \coloneq \sum_{a \in A_{+}, \, \deg a = d} \dfrac{1}{a^{s}} \in k.
\end{align*}
Let $\fs = (s_{1}, \ldots, s_{r}) \in \cI$.
The multiple zeta value in positive characteristic is defined by Thakur \cite[Definition~5.10.1]{T04} as
\begin{align*}
    \zeta_{A}(\fs) \coloneq \sum_{d_{1} > \cdots > d_{r} \geq 0} S_{d_{1}}(s_{1}) \cdots S_{d_{r}}(s_{r}) = \sum_{\substack{a_{i} \in A_{+} \\ \deg a_{1} > \cdots > \deg a_{r}}} \dfrac{1}{a_{1}^{s_{1}} \cdots a_{r}^{s_{r}}} \in k_{\infty},
\end{align*}
and the Carlitz multiple polylogarithm (CMPL) is defined by Chang \cite[Definition~5.1.1]{C14} as
\begin{align*}
    \Li_{\fs}(z_{1}, \ldots, z_{r}) \coloneq \sum_{d_{1} > \cdots > d_{r} \geq 0} \dfrac{z_{1}^{q^{d_{1}}} \cdots z_{r}^{q^{d_{r}}}}{L_{d_{1}}^{s_{1}} \cdots L_{d_{r}}^{s_{r}}} \in k[\![z_{1}, \ldots, z_{r}]\!],
\end{align*}
with the convention that $\zeta_{A}(\emptyset) = \Li_{\emptyset}(z_{1}, \ldots, z_{0}) = 1$.
These are function field analogues of real-valued MZVs and multiple polylogarithms, respectively.
In $\CC_{\infty}$, $\Li_{\fs}(z_{1}, \ldots, z_{r})$ converges when $|z_{1}|_{\infty} < q^{\frac{s_{1} q}{q - 1}}$ and $|z_{i}|_{\infty} \leq q^{\frac{s_{i} q}{q - 1}}$ ($2 \leq i \leq r$).
In particular, it converges at $(1, \ldots, 1)$ in $k_{\infty}$, and we set
\begin{align*}
    \cL(\fs) \coloneq \Li_{\fs}(1, \ldots, 1) = \sum_{d_{1} > \cdots > d_{r} \geq 0} \dfrac{1}{L_{d_{1}}^{s_{1}} \cdots L_{d_{r}}^{s_{r}}} \in k_{\infty}.
\end{align*}
Although $S_{d}(s)$ and $L_{d}^{- s}$ are different in general, they coincide for $s \leq q$ (see \cite[Theorem~5.9.1]{T04}).
In particular, we have $\zeta_{A}(\fs) = \cL(\fs)$ whenever $s_{1}, \ldots, s_{r} \leq q$.
As a function field analogue of Euler's result, Carlitz \cite[Theorem~9.3]{Ca35} proved that
\begin{align} \label{equation-Euler-Carlitz}
    \zeta_{A}((q - 1) n) \in \Fp(L_{1})^{\times} \cdot \tpi^{(q - 1) n}
\end{align}
for all $n \in \ZZ_{\geq 1}$, where $\tpi \in \sqrt[q - 1]{- \theta} \cdot  k_{\infty}^{\times}$ is the Carlitz period.

For each $w \in \ZZ_{\geq 0}$, let $\cZ_{R, w} \subset \CC_{\infty}$ be the $R$-module spanned by the values $\zeta_{A}(\fs)$ for $\fs \in \cI$ with $\wt(\fs) = w$, and set $\cZ_{R} \coloneq \sum_{w \geq 0} \cZ_{R, w}$.
Thakur \cite[Theorem~3]{T10} proved that, in positive characteristic, the product of MZVs of weights $w$ and $w'$ can be written as a finite sum of MZVs of weight $w + w'$; this is called the $q$-shuffle product formula.
Chang \cite[Theorem~2.2.1]{C14} proved that the natural surjection
\begin{align*}
    \ok \otimes_{k} \Biggl( \bigoplus_{w \geq 0} \cZ_{k, w} \Biggr) \twoheadrightarrow \cZ_{\ok}
\end{align*}
is an isomorphism.
In particular, the subspace $\cZ_{R} \subset \CC_{\infty}$ forms a graded $R$-algebra.

Let $\IT \subset \cI$ be the subset of Thakur's indices, i.e.,
\begin{align*}
    \IT \coloneq \{ (s_{1}, \ldots, s_{r}) \in \cI \, | \, r \in \ZZ_{\geq 0}, \ s_{1}, \ldots, s_{r - 1} \leq q, \ s_{r} \leq q - 1 \}.
\end{align*}
We note that the index $\emptyset$ is an element of $\IT$.
Thakur \cite[Conjecture~8.2]{T17} conjectured that the values $\zeta_{A}(\fs)$ ($\fs \in \IT$) form a basis of $\cZ_{k}$.
This is a refinement of Todd's dimension conjecture \cite[Conjecture~7.1]{To18}.
Ngo Dac \cite[Theorem~A]{ND21} proved that these values generate all the MZVs over $\Fp[L_{1}]$.
Subsequently, Chang, Chen, and Mishiba \cite[Theorem~1.5]{CCM23}, as well as Im, Kim, Le, Ngo Dac, and Pham \cite[Theorem~B]{IKLNDP24}, proved that Thakur's conjecture is true.
Therefore, the values $\zeta_{A}(\fs)$ ($\fs \in \IT$) form a basis of $\cZ_{R}$.
Note that they also showed that
\begin{align*}
    \cL(\fs) \in \cZ_{R, \wt(\fs)}
\end{align*}
for all $\fs \in \cI$.
In particular, the $R$-module spanned by $\cL(\fs)$ ($\fs \in \cI$) is also $\cZ_{R}$.

\subsection{Main theorem}

\begin{definition}
For each $\fs = (s_{1}, \ldots, s_{r}) \in \cI$, we define the $*$-inverse multiple zeta value $\zetainv_{A}(\fs)$, the $*$-inverse Carlitz multiple polylogarithm $\Liinv_{\fs}(z_{1}, \ldots, z_{r})$, and its special value $\cLinv(\fs)$ by
\begin{align*}
    &\zetainv_{A}(\fs) \coloneq (- 1)^{r} \sum_{0 \leq d_{1} \leq \cdots \leq d_{r}} S_{d_{1}}(s_{1}) \cdots S_{d_{r}}(s_{r}) = (- 1)^{r} \sum_{\substack{a_{i} \in A_{+} \\ \deg a_{1} \leq \cdots \leq \deg a_{r}}} \dfrac{1}{a_{1}^{s_{1}} \cdots a_{r}^{s_{r}}} \in k_{\infty}, \\
    &\Liinv_{\fs}(z_{1}, \ldots, z_{r}) \coloneq (- 1)^{r} \sum_{0 \leq d_{1} \leq \cdots \leq d_{r}} \dfrac{z_{1}^{q^{d_{1}}} \cdots z_{r}^{q^{d_{r}}}}{L_{d_{1}}^{s_{1}} \cdots L_{d_{r}}^{s_{r}}} \in k[\![z_{1}, \ldots, z_{r}]\!], \\
    &\cLinv(\fs) \coloneq \Liinv_{\fs}(1, \ldots, 1) = (- 1)^{r} \sum_{0 \leq d_{1} \leq \cdots \leq d_{r}} \dfrac{1}{L_{d_{1}}^{s_{1}} \cdots L_{d_{r}}^{s_{r}}} \in k_{\infty}.
\end{align*}
\end{definition}

For clarity, we call the original objects the non-$*$-inverse ones.
It follows from the definitions that
\begin{align*}
    \zetainv_{A}(\fs) &= (- 1)^{\dep(\fs)} \sum_{\fn} \zeta_{A}(\fn) \in \cZ_{R, \wt(\fs)}, \\
    \cLinv(\fs) &= (- 1)^{\dep(\fs)} \sum_{\fn} \cL(\fn) \in \cZ_{R, \wt(\fs)}
\end{align*}
for all $\fs = (s_{1}, \ldots, s_{r})\in \cI$, where the sums run over all indices $\fn$ of the form $\fn = (s_{r} \, \square_{r - 1} \, \cdots \, \square_{2} \, s_{2} \, \square_{1} \, s_{1})$, with each $\square_{i}$ being either ``$+$'' or ``$,$''.
As mentioned in Subsection~\ref{section-positive-characteristic}, we have $S_{d}(s) = L_{d}^{- s}$ for all $d \in \ZZ_{\geq 0}$ and all $s \in \ZZ_{\geq 1}$ with $s \leq q$.
Hence $\zetainv_{A}(\fs) = \cLinv(\fs)$ for all $\fs = (s_{1}, \ldots, s_{r}) \in \cI$ with $s_{1}, \ldots, s_{r} \leq q$.

\begin{remark}
The $*$-inverse values are obtained via the antipode of a certain Hopf algebra, denoted by $\fh^{1}_{R}$ (see \S~\ref{subsection-algebraic-setup}).
This terminology is in analogy with the inversed MZVs introduced by Furusho in \cite{F20} (see Remark~\ref{remark-inversed-MZV}).
These $*$-inverse values also arise naturally in the study of $t$-modules and $t$-motives; see \cite{CM19}, \cite{CM21}, and \cite{GND23}.
Note that, in these works, the authors instead considered the multiple zeta star values (MZSVs)
\begin{align*}
    \zeta^{\star}_{A}(s_{1}, s_{2}, \ldots, s_{r}) \coloneq (- 1)^{r} \zetainv_{A}(s_{r}, \ldots, s_{2}, s_{1})
\end{align*}
and the Carlitz multiple star polylogarithms (CMSPLs)
\begin{align*}
    \Lis_{(s_{1}, s_{2}, \ldots, s_{r})}(z_{1}, z_{2}, \ldots, z_{r}) \coloneq (- 1)^{r} \Liinv_{(s_{r}, \ldots, s_{2}, s_{1})}(z_{r}, \ldots, z_{2}, z_{1}).
\end{align*}
\end{remark}

\begin{example}
For lower depth cases,
\begin{align*}
    &\zetainv_{A}(\emptyset) = \zeta_{A}(\emptyset) = 1, \ \
    \zetainv_{A}(s_{1}) = - \zeta_{A}(s_{1}), \ \
    \zetainv_{A}(s_{1}, s_{2}) = \sum_{0 \leq d_{1} \leq d_{2}} S_{d_{1}}(s_{1}) S_{d_{2}}(s_{2}), \\
    &\cLinv(\emptyset) = \cL(\emptyset) = 1, \ \
    \cLinv(s_{1}) = - \cL(s_{1}), \ \
    \cLinv(s_{1}, s_{2}) = \sum_{0 \leq d_{1} \leq d_{2}} \dfrac{1}{L_{d_{1}}^{s_{1}} L_{d_{2}}^{s_{2}}}.
\end{align*}
\end{example}

In the following, $\zeta_{A}(q - 1) \cZ_{R}$ denotes the homogeneous ideal of $\cZ_{R} = \bigoplus_{w \geq 0} \cZ_{R, w}$ generated by
\begin{align*}
    \zeta_{A}(q - 1) = \cL(q - 1) = - \zetainv_{A}(q - 1) = - \cLinv(q - 1).
\end{align*}
Let
\begin{align*}
    \ocZ_{R} \coloneq \cZ_{R} / \zeta_{A}(q - 1) \cZ_{R}
    = \biggl( \bigoplus_{0 \leq w \leq q - 2} \cZ_{R, w} \biggr) \oplus \biggl( \bigoplus_{w \geq q - 1} \cZ_{R, w} / \zeta_{A}(q - 1) \cZ_{R, w - (q - 1)} \biggr)
\end{align*}
be the graded quotient $R$-algebra of $\cZ_{R}$ by $\zeta_{A}(q - 1) \cZ_{R}$.
For each $x \in \cZ_{R}$, let
\begin{align*}
    x \bmod \zeta_{A}(q - 1) \cZ_{R}
\end{align*}
denote the class of $x$ in $\ocZ_{R}$.
We set
\begin{align*}
    \ocL(\fs) \coloneq \cL(\fs) \bmod \zeta_{A}(q - 1) \cZ_{R}
    \ \ \ \textrm{and} \ \ \
    \ocLinv(\fs) \coloneq \cLinv(\fs) \bmod \zeta_{A}(q - 1) \cZ_{R}
\end{align*}
for each $\fs \in \cI$.
The following is the main theorem of the present paper.

\begin{theorem} \label{theorem-main}
There exist $R$-algebra homomorphisms
\begin{align*}
    \Delta_{\cZ} \colon \cZ_{R} \to \cZ_{R} \otimes_{R} \overline{\cZ}_{R}, \ \
    \overline{\Delta} \colon \overline{\cZ}_{R} \to \overline{\cZ}_{R} \otimes_{R} \overline{\cZ}_{R}, \ \
    \overline{\varepsilon} \colon \overline{\cZ}_{R} \to R, \ \
    \overline{S} \colon \overline{\cZ}_{R} \to \overline{\cZ}_{R}
\end{align*}
such that
\begin{align*}
    \Delta_{\cZ}(\cL(\fs)) &= \sum_{\substack{\fs_{1}, \fs_{2} \in \cI \\ (\fs_{1}, \fs_{2}) = \fs}} \cL(\fs_{1}) \otimes \overline{\cL}(\fs_{2}), \\
    \overline{\Delta}(\overline{\cL}(\fs)) &= \sum_{\substack{\fs_{1}, \fs_{2} \in \cI \\ (\fs_{1}, \fs_{2}) = \fs}} \overline{\cL}(\fs_{1}) \otimes \overline{\cL}(\fs_{2}), \\
    \overline{\varepsilon}(\overline{\cL}(\fs)) &= \left\{ \begin{array}{@{}ll} 1 & (\fs = \emptyset) \\ 0 & (\fs \neq \emptyset) \end{array} \right., \\
    \overline{S}(\ocL(\fs)) &= \ocLinv(\fs)
\end{align*}
for all $\fs \in \cI$.
Moreover, the $R$-algebra $\overline{\cZ}_{R}$ admits the structure of a connected graded involutive Hopf algebra with coproduct $\overline{\Delta}$, counit $\overline{\varepsilon}$, and non-trivial antipode $\overline{S}$, and $\cZ_{R}$ admits the structure of a graded right comodule over $\overline{\cZ}_{R}$ with coaction $\Delta_{\cZ}$.
\end{theorem}

By Theorem~\ref{theorem-main}, the two collections $\{ \cL(\fs) \}_{\fs \in \cI}$ and $\{ \cLinv(\fs) \}_{\fs \in \cI}$ satisfy the same algebraic relations modulo $\zeta_{A}(q - 1) \cZ_{R}$.

\begin{remark}
(1)
Theorem~\ref{theorem-main} can be regarded as a function field analogue of the Hopf algebra and comodule structures arising from motivic multiple zeta values in characteristic zero, as explained in \S~\ref{subsection-motivation}.
We emphasize that, in our setting, these structures are well-defined directly on the spaces spanned by the values.

\noindent
(2)
Since $\overline{S}$ is an involution on $\ocZ_{R}$, we have
\begin{align*}
\overline{S}(\ocLinv(\fs)) = \ocL(\fs),
\end{align*}
and hence
\begin{align*}
\overline{S}\bigl(\Lis_{\fs}(1, \ldots, 1) \bmod \zeta_{A}(q - 1) \cZ_{R}\bigr) = (- 1)^{\dep(\fs)} \ocL(s_{r}, \ldots, s_{2}, s_{1})
\end{align*}
for all $\fs = (s_{1}, \ldots, s_{r}) \in \cI$.
\end{remark}

\begin{example}
Thakur \cite[Theorem~5]{T09b} proved that the fundamental relation
\begin{align*}
    \cL(q) - L_{1} \cL(1, q - 1) = 0
\end{align*}
holds.
Therefore, according to Theorem~\ref{theorem-main}, the values $\cLinv(\fs)$ satisfy the same relation modulo $\zeta_{A}(q - 1) \cZ_{R}$.
Indeed, by the harmonic product formula \eqref{equation-product-formula}, we have
\begin{align*}
    \cL(1) \cL(q - 1) = \cL(1, q - 1) + \cL(q - 1, 1) + \cL(q) = \cL(1, q - 1) + \cLinv(1, q - 1).
\end{align*}
Therefore,
\begin{align*}
    \cLinv(q) - L_{1} \cLinv(1, q - 1) = - L_{1} \cL(1) \cL(q - 1) \equiv 0 \bmod \zeta_{A}(q - 1) \cZ_{R}.
\end{align*}
\end{example}

The following conjecture arises naturally, but it is still open.

\begin{conjecture} \label{conjecture-zeta-image}
Assume that $\Fp(L_{1}) \subset R$.
Then
\begin{align*}
    &\Delta_{\cZ}(\zeta_{A}(\fs)) = \sum_{\substack{\fs_{1}, \fs_{2} \in \cI \\ (\fs_{1}, \fs_{2}) = \fs}} \zeta_{A}(\fs_{1}) \otimes \bigl( \zeta_{A}(\fs_{2}) \bmod \zeta_{A}(q - 1) \cZ_{R} \bigr), \\
    &\overline{S} \bigl( \zeta_{A}(\fs) \bmod \zeta_{A}(q - 1) \cZ_{R} \bigr) = \zetainv_{A}(\fs) \bmod \zeta_{A}(q - 1) \cZ_{R}
\end{align*}
for all $\fs \in \cI$.
\end{conjecture}

\begin{remark}
The necessity of the assumption $\Fp(L_{1}) \subset R$ in Conjecture~\ref{conjecture-zeta-image} was pointed out by Tsui.
See Propositions~\ref{proposition-zeta-*-inv-prod}, \ref{proposition-zeta-A}, and Corollary~\ref{corollary-zeta-U}.
Tsui \cite{Tsui26} also announced a proof of Conjecture~\ref{conjecture-zeta-image}.
\end{remark}

Let $v \in A$ be a monic irreducible polynomial.
We now discuss the $v$-adic setting and formulate a conjectural $v$-adic analogue of Theorem~\ref{theorem-main}.
In \cite[Definition~6.1.1]{CM21}, Chang and Mishiba defined the $v$-adic MZVs $\zeta_{A}(\fs)_{v}$ ($\fs \in \cI$) as function field analogues of Furusho's $p$-adic MZVs \cite[Definition~2.17]{F04}.
In \cite[Theorem~1.2.3]{CCM22}, Chang, Chen, and Mishiba proved the existence of a well-defined $R$-algebra homomorphism
\begin{align*}
    \cZ_{R} / \zeta_{A}(q - 1) \cZ_{R} \twoheadrightarrow \cZ_{v, R}
\end{align*}
given by $\zeta_{A}(\fs) \bmod \zeta_{A}(q - 1) \cZ_{R} \mapsto \zeta_{A}(\fs)_{v}$, where $\cZ_{v, R}$ is the $R$-algebra generated by all $v$-adic MZVs.
Moreover, they conjectured that this surjection is an isomorphism (when $R = \ok$).
Therefore, we have the following conjecture:
\begin{conjecture}
The $R$-algebra $\cZ_{v, R}$ admits a connected graded involutive Hopf algebra structure.
In particular, the map $\overline{S}$ induces a non-trivial $R$-algebra involution on $\cZ_{v, R}$.
\end{conjecture}

Finally, in the context of the harmonic product, we briefly mention characteristic zero statements formally analogous to Theorem~\ref{theorem-main} and Conjecture~\ref{conjecture-zeta-image}.
Since the ideal $\zeta_{A}(q - 1) \cZ_{R}$ is a function field analogue of $\zeta(2) \fZ_{\QQ}$, it is natural to ask whether the antipode of the Hopf algebra structure on $\fh^{1}_{\QQ}$ given by the harmonic product and the deconcatenation coproduct induces an involution on $\fZ_{\QQ}/ \zeta(2) \fZ_{\QQ}$.
However, the corresponding statements may not hold.
We explain this in \S~\ref{subsection-*-inverse-MZV-char-zero}.

\subsection{Organization of the paper}

In Section~\ref{section-setup-relations}, we introduce the algebraic setup and review several known relations among MZVs and the special values of CMPLs in positive characteristic.
We also establish the harmonic product formula for the special values of $*$-inverse CMPLs and relations between non-$*$-inverse and $*$-inverse values.
In addition, we review linear relations among non-$*$-inverse values that will be used later.

In Section~\ref{section-proof}, we establish the Hopf algebra and comodule structures stated in Theorem~\ref{theorem-main}.
The proof is based on a key lemma established in Subsection~\ref{subsection-key-lemma}.

In Section~\ref{section-*-inverse-MZV}, we study relations among $*$-inverse MZVs.
We first establish, modulo $\zeta_A(q-1)\cZ_R$, the $q$-shuffle product formula for $*$-inverse single zeta values and certain linear relations among $*$-inverse MZVs.
These results provide evidence for Conjecture~\ref{conjecture-zeta-image}.
Finally, we consider $*$-inverse MZVs in characteristic zero.
We observe that a naive analogue based on the Hopf algebra structure given by the harmonic product and the deconcatenation coproduct may fail and formulate a modified problem.

\section{Algebraic preliminaries and relations} \label{section-setup-relations}

\subsection{Algebraic setup} \label{subsection-algebraic-setup}

Let $\fh^{1}_{R}$ be the free $R$-algebra on the set $\ZZ_{\geq 1}$.
We regard $\cI$ as the $R$-basis of $\fh^{1}_{R}$.
The product of $P_{1}, \ldots, P_{m} \in \fh^{1}_{R}$ is denoted by $(P_{1}, \ldots, P_{m})$.
When $P_{i} = \fs_{i} \in \cI$, the product $(\fs_{1}, \ldots, \fs_{m})$ coincides with the concatenation of $\fs_{1}, \ldots, \fs_{m}$.
For each $w \in \ZZ_{\geq 0}$, let $\cI_{w}$ be the set of indices of weight $w$, and let $\ITw \coloneq \IT \cap \cI_{w}$.
Let $\fh^{1}_{R, w} \subset \fh^{1}_{R}$ be the $R$-submodule spanned by $\cI_{w}$.
Therefore, we have a weight decomposition $\fh^{1}_{R} = \bigoplus_{w \geq 0} \fh^{1}_{R, w}$.
For each $P = \sum_{\fs \in \cI} a_{\fs} \fs \in \fh^{1}_{R}$ ($a_{\fs} \in R$), we set $\Supp(P) \coloneq \{ \fs \in \cI \, | \, a_{\fs} \neq 0 \}$.

Let $\fs = (s_{1}, s_{2}, \ldots, s_{r}) \in \cI$.
For each $1 \leq i \leq r$, we set
\begin{align*}
    \fs[\,:i] \coloneq (s_{1}, s_{2}, \ldots, s_{i}), \ \ \
    \fs[i:\,] \coloneq (s_{i}, s_{i + 1}, \ldots, s_{r}), \ \ \
    \fs[\,:0] \coloneq \fs[r + 1:\,] \coloneq \emptyset.
\end{align*}
In the following, we interpret $(s_{i}, \ldots, s_{j}) = \emptyset$ when $i > j$.
When $\fs \neq \emptyset$, we set
\begin{align*}
    \fs_{+} \coloneq \fs[\,:\dep(\fs) - 1] = (s_{1}, \ldots, s_{r - 1}) \ \ \textrm{and} \ \ \fs_{-} \coloneq \fs[2:\,] = (s_{2}, \ldots, s_{r}).
\end{align*}

For each $s, n, j \in \ZZ_{\geq 1}$, we define an integer $\Delta_{s, n}^{[j]}$ by
\begin{align*}
    \Delta_{s, n}^{[j]} \coloneq \left\{ \begin{array}{@{}ll} \displaystyle (- 1)^{s - 1} \binom{j - 1}{s - 1} + (- 1)^{n - 1} \binom{j - 1}{n - 1} & \textrm{if $(q - 1) \mid j$ and $1 \leq j < s + n$} \\ 0 & \textrm{otherwise} \end{array} \right..
\end{align*}
Let $*^{\zeta} \colon (\fh^{1}_{R})^{2} \to \fh^{1}_{R}$ and $* = *^{\Li} \colon (\fh^{1}_{R})^{2} \to \fh^{1}_{R}$ be the $q$-shuffle product and harmonic product, respectively, that is, the $R$-bilinear maps such that
\begin{align*}
    &\emptyset *^{\zeta} P = P *^{\zeta} \emptyset = \emptyset * P = P * \emptyset = P, \\
    &\fs *^{\zeta} \fn = (s_{1}, \fs_{-} *^{\zeta} \fn) + (n_{1}, \fs *^{\zeta} \fn_{-}) + (s_{1} + n_{1}, \fs_{-} *^{\zeta} \fn_{-}) + D_{\fs}(\fn), \\
    &\fs * \fn = (s_{1}, \fs_{-} * \fn) + (n_{1}, \fs * \fn_{-}) + (s_{1} + n_{1}, \fs_{-} * \fn_{-})
\end{align*}
for all $P \in \fh^{1}_{R}$ and $\fs = (s_{1}, \ldots, s_{r}), \fn = (n_{1}, \ldots, n_{\ell}) \in \cI \setminus \{ \emptyset \}$, where we define
\begin{align*}
    D_{\fs}(\fn) \coloneq \sum_{j = 1}^{s_{1} + n_{1} - 1} \Delta_{s_{1}, n_{1}}^{[j]} (s_{1} + n_{1} - j, (j) *^{\zeta} (\fs_{-} *^{\zeta} \fn_{-})).
\end{align*}
We set $D_{\fs}(\emptyset) \coloneq 0$, extend $D_{\fs}$ to an $R$-linear map $\fh^{1}_{R} \to \fh^{1}_{R}$, and denote this extension again by $D_{\fs}$.
We also define $R$-linear maps $\eta \colon R \to \fh^{1}_{R}$, $\Delta \colon \fh^{1}_{R} \to \fh^{1}_{R} \otimes_{R} \fh^{1}_{R}$, and $\varepsilon \colon \fh^{1}_{R} \to R$ by
\begin{align*}
    \eta(1) \coloneq \emptyset, \ \ \
    \Delta(\fs) \coloneq \sum_{\substack{\fs_{1}, \fs_{2} \in \cI \\ (\fs_{1}, \fs_{2}) = \fs}} \fs_{1} \otimes \fs_{2}, \ \ \
    \varepsilon(\fs) \coloneq \left\{ \begin{array}{@{}ll} 1 & (\fs = \emptyset) \\ 0 & (\fs \neq \emptyset) \end{array} \right.
\end{align*}
for all $\fs \in \cI$.
Hoffman \cite[Section~3]{Ho00} showed that $(\fh^{1}_{R}, *, \eta, \Delta, \varepsilon)$ is a Hopf algebra over $R$ with antipode $S \colon \fh^{1}_{R} \to \fh^{1}_{R}$ given by
\begin{align} \label{equation-antipode}
    S(\fs) \coloneq \sum_{\substack{(\fs_{1}, \ldots, \fs_{m}) = \fs \\ (m \geq 0, \, \fs_{i} \in \cI \setminus \{ \emptyset \})}} (- 1)^{m} \fs_{1} * \cdots * \fs_{m}
    = (- 1)^{\dep(\fs)} \sum_{\fn} \fn
\end{align}
for all $\fs = (s_{1}, \ldots, s_{r})\in \cI$, where in the second expression the sum runs over all indices $\fn$ of the form $\fn = (s_{r} \, \square_{r - 1} \, \cdots \, \square_{2} \, s_{2} \, \square_{1} \, s_{1})$, with each $\square_{i}$ being either ``$+$'' or ``$,$''.
We note that in \cite[Section~3]{Ho00} the base ring is assumed to be a subfield of $\CC$, but the same argument works over an arbitrary commutative ring.
In our notation, the antipode axiom takes the following form for $\fs \in \cI$:
\begin{align} \label{equation-antipode-axiom}
    \sum_{\substack{\fs_{1}, \fs_{2} \in \cI \\ (\fs_{1}, \fs_{2}) = \fs}} \fs_{1} * S(\fs_{2})
    = \sum_{\substack{\fs_{1}, \fs_{2} \in \cI \\ (\fs_{1}, \fs_{2}) = \fs}} S(\fs_{1}) * \fs_{2}
    = \left\{ \begin{array}{@{}ll} \emptyset & (\fs = \emptyset) \\ 0 & (\fs \neq \emptyset) \end{array} \right..
\end{align}

\begin{remark}
In \cite[Section~3.2.3]{Sh18}, Shi conjectured that $\fh^{1}_{R}$ admits another Hopf algebra structure whose product is $*^{\zeta}$.
Subsequently, a proof of this conjecture was announced by Im, Kim, Le, Ngo Dac, and Pham in \cite{IKLNDP23}.
\end{remark}

The following elementary lemma will be used to prove the product formula for $\cLinv(\fs)$.

\begin{lemma} \label{lemma-harmonic-reverse}
For each $\fs = (s_{1}, \ldots, s_{r}), \fn = (n_{1}, \ldots, n_{\ell}) \in \cI \setminus \{ \emptyset \}$, we have
\begin{align*}
    \fs * \fn = (\fs_{+} * \fn, s_{r}) + (\fs * \fn_{+}, n_{\ell}) + (\fs_{+} * \fn_{+}, s_{r} + n_{\ell}).
\end{align*}
\end{lemma}

\begin{proof}
The lemma is proved by induction on $\dep(\fs) + \dep(\fn)$.
When $\dep(\fs) = \dep(\fn) = 1$, the lemma is immediate.
Let $\dep(\fs) + \dep(\fn) \geq 3$ and assume that the lemma holds for indices whose total depth is less than $\dep(\fs) + \dep(\fn)$.
When $\dep(\fn) = 1$, we have $\dep(\fs) \geq 2$ and
\begin{align*}
    \fs * (n_{1}) &= (s_{1}, \fs_{-} * (n_{1})) + (n_{1}, \fs) + (s_{1} + n_{1}, \fs_{-}) \\
    &= (s_{1}, \fs_{\pm} * (n_{1}), s_{r}) + (s_{1}, \fs_{-}, n_{1}) + (s_{1}, \fs_{\pm}, s_{r} + n_{1}) + (n_{1}, \fs) + (s_{1} + n_{1}, \fs_{-}) \\
    &= (\fs_{+} * (n_{1}), s_{r}) + (\fs, n_{1}) + (\fs_{+}, s_{r} + n_{1}),
\end{align*}
where $\fs_{\pm} \coloneq (\fs_{+})_{-} = (\fs_{-})_{+} = (s_{2}, \ldots, s_{r - 1})$.
The case $\dep(\fs) = 1$ is proved similarly.
Finally, when $\dep(\fs), \dep(\fn) \geq 2$, we have
\begin{align*}
    \fs * \fn
    &= (s_{1}, \fs_{-} * \fn) + (n_{1}, \fs * \fn_{-}) + (s_{1} + n_{1}, \fs_{-} * \fn_{-}) \\
    &= (s_{1}, \fs_{\pm} * \fn, s_{r})
    + (s_{1}, \fs_{-} * \fn_{+}, n_{\ell}) + (s_{1}, \fs_{\pm} * \fn_{+}, s_{r} + n_{\ell}) \\
    &\phantom{=} \; + (n_{1}, \fs_{+} * \fn_{-}, s_{r}) + (n_{1}, \fs * \fn_{\pm}, n_{\ell}) + (n_{1}, \fs_{+} * \fn_{\pm}, s_{r} + n_{\ell}) \\
    &\phantom{=} \; + (s_{1} + n_{1}, \fs_{\pm} * \fn_{-}, s_{r})
    + (s_{1} + n_{1}, \fs_{-} * \fn_{\pm}, n_{\ell}) + (s_{1} + n_{1}, \fs_{\pm} * \fn_{\pm}, s_{r} + n_{\ell}) \\
    &= (\fs_{+} * \fn, s_{r}) + (\fs * \fn_{+}, n_{\ell}) + (\fs_{+} * \fn_{+}, s_{r} + n_{\ell}).
\end{align*}
\end{proof}

\subsection{Algebraic relations}

We extend the maps $\cI \ni \fs \mapsto \zeta_{A}(\fs), \zetainv_{A}(\fs), \cL(\fs), \cLinv(\fs) \in \cZ_{R}$ to $R$-linear maps on $\fh^{1}_{R}$, and denote them by the same symbols.
Then we have
\begin{align} \label{equation-*-inverse-antipode}
    \zetainv_{A}(P) = \zeta_{A}(S(P))
    \ \ \textrm{and} \ \
    \cLinv(P) = \cL(S(P))
\end{align}
for all $P \in \fh^{1}_{R}$.
To treat ($*$-inverse) MZVs and the special values of ($*$-inverse) CMPLs simultaneously, we define
\begin{align*}
    &\sLz(P) \coloneq \zeta_{A}(P), \ \ \
    \sLzinv(P) \coloneq \zetainv_{A}(P), \\
    &\sLL(P) \coloneq \cL(P), \ \ \
    \sLLinv(P) \coloneq \cLinv(P)
\end{align*}
for each $P \in \fh^{1}_{R}$.
Then
\begin{align} \label{equation-product-formula}
    \sL^{\bullet}(P) \sL^{\bullet}(Q) = \sL^{\bullet}(P *^{\bullet} Q)
\end{align}
for all $\bullet \in \{ \zeta, \Li \}$ and all $P, Q \in \fh^{1}_{R}$ (see, e.g., \cite{C14}, \cite{Ch15}, \cite{T10}, and \cite{T17}).

We next establish the harmonic product formula for the values $\cLinv(\fs)$.

\begin{proposition} \label{proposition-Li-*-inv-prod}
The values $\cLinv(P)$ ($P \in \fh^{1}_{R}$) satisfy the harmonic product formula; that is, for all $P, Q \in \fh^{1}_{R}$, we have $\cLinv(P) \cLinv(Q) = \cLinv(P * Q)$.
\end{proposition}

\begin{proof}
For each $D \in \ZZ_{\geq 0}$, let $\cLinv_{\leq D} \colon \fh^{1}_{R} \to \ok$ be the $R$-linear map such that
\begin{align*}
    \cLinv_{\leq D}(\fa) = (- 1)^{\dep(\fa)} \sum_{0 \leq d_{1} \leq \cdots \leq d_{m} \leq D} \dfrac{1}{L_{d_{1}}^{a_{1}} \cdots L_{d_{m}}^{a_{m}}}
\end{align*}
for all $\fa = (a_{1}, \ldots, a_{m}) \in \cI$.
We claim that $\cLinv_{\leq D}(\fs) \cLinv_{\leq D}(\fn) = \cLinv_{\leq D}(\fs * \fn)$ for all $D \in \ZZ_{\geq 0}$ and all $\fs = (s_{1}, \ldots, s_{r}), \fn = (n_{1}, \ldots, n_{\ell}) \in \cI$.
We show the claim by induction on $\dep(\fs) + \dep(\fn)$.
When $\fs = \emptyset$ or $\fn = \emptyset$, the claim is clear.
When $\fs, \fn \neq \emptyset$, we assume that the claim is valid for indices whose total depth is less than $\dep(\fs) + \dep(\fn)$.
Then by Lemma~\ref{lemma-harmonic-reverse},
\begin{align*}
    &\cLinv_{\leq D}(\fs) \cLinv_{\leq D}(\fn) \\
    &= \sum_{0 \leq d \leq D} \left( - \dfrac{\cLinv_{\leq d}(\fs_{+}) \cLinv_{\leq d}(\fn)}{L_{d}^{s_{r}}} - \dfrac{\cLinv_{\leq d}(\fs) \cLinv_{\leq d}(\fn_{+})}{L_{d}^{n_{\ell}}} - \dfrac{\cLinv_{\leq d}(\fs_{+}) \cLinv_{\leq d}(\fn_{+})}{L_{d}^{s_{r} + n_{\ell}}} \right) \\
    &= \sum_{0 \leq d \leq D} \left( - \dfrac{\cLinv_{\leq d}(\fs_{+} * \fn)}{L_{d}^{s_{r}}} - \dfrac{\cLinv_{\leq d}(\fs * \fn_{+})}{L_{d}^{n_{\ell}}} - \dfrac{\cLinv_{\leq d}(\fs_{+} * \fn_{+})}{L_{d}^{s_{r} + n_{\ell}}} \right) \\
    &= \cLinv_{\leq D}(\fs_{+} * \fn, s_{r}) + \cLinv_{\leq D}(\fs * \fn_{+}, n_{\ell}) + \cLinv_{\leq D}(\fs_{+} * \fn_{+}, s_{r} + n_{\ell}) \\
    &=\cLinv_{\leq D}(\fs * \fn).
\end{align*}
By the claim, we have $\cLinv_{\leq D}(P) \cLinv_{\leq D}(Q) = \cLinv_{\leq D}(P * Q)$ for all $D \in \ZZ_{\geq 0}$ and all $P, Q \in \fh^{1}_{R}$.
Taking the limit as $D \to \infty$, we obtain the desired result.
\end{proof}

The following proposition gives a basic relation between the non-$*$-inverse values and the $*$-inverse values.

\begin{proposition} \label{proposition-prod-sum}
For each $\bullet \in \{ \zeta, \Li \}$ and each $\fs \in \cI \setminus \{ \emptyset \}$, we have
\begin{align*}
    \sum_{i = 0}^{\dep(\fs)} \sLb(\fs[\,:i]) \sLbinv(\fs[i + 1:\,]) = 0
    \ \ \ \textrm{and} \ \ \
    \sum_{i = 0}^{\dep(\fs)} \sLbinv(\fs[\,:i]) \sLb(\fs[i + 1:\,]) = 0.
\end{align*}
\end{proposition}

\begin{proof}
We first note that when $\bullet = \Li$, both formulas follow immediately from \eqref{equation-antipode-axiom}, \eqref{equation-*-inverse-antipode}, and \eqref{equation-product-formula}.
A more general version of this proposition was stated in \cite[Lemma~4.1]{GND23}, and a proof of the $v$-adic analogue of the first equation for CMPLs was given in \cite[Lemma~4.2.1]{CM19}.
For analogous results in characteristic zero, see \cite[Proposition~6]{IKOO11}, \cite[Theorem~2.13]{SS17}, and \cite[Theorem~3]{Zl05}.

Here, we give a unified proof only of the second equation for $\bullet = \zeta$ and $\bullet = \Li$.
For each $s \in \ZZ_{\geq 1}$ and $d \in \ZZ_{\geq 0}$, we set
\begin{align*}
    S^{\zeta}_{d}(s) \coloneq S_{d}(s) \ \ \textrm{and} \ \ S^{\Li}_{d}(s) \coloneq \dfrac{1}{L_{d}^{s}}.
\end{align*}
Let $\fs = (s_{1}, \ldots, s_{r}) \in \cI \setminus \{ \emptyset \}$.
Then for each $1 \leq i \leq r - 1$,
\begin{align*}
    (- 1)^{i} \sum_{\substack{0 \leq d_{1} \leq \cdots \leq d_{i} \\ d_{i} > \cdots > d_{r} \geq 0}} S^{\bullet}_{d_{1}}(s_{1}) \cdots S^{\bullet}_{d_{r}}(s_{r})
    &= \sLbinv(\fs[\,:i]) \sLb(\fs[i + 1:\,]) \\
    &\phantom{=} \; + (- 1)^{i + 1} \sum_{\substack{0 \leq d_{1} \leq \cdots \leq d_{i + 1} \\ d_{i + 1} > \cdots > d_{r} \geq 0}} S^{\bullet}_{d_{1}}(s_{1}) \cdots S^{\bullet}_{d_{r}}(s_{r}).
\end{align*}
Therefore, we have
\begin{align*}
    - \sLb(\fs)
    &= (- 1)^{1} \sum_{\substack{0 \leq d_{1} \\ d_{1} > \cdots > d_{r} \geq 0}} S^{\bullet}_{d_{1}}(s_{1}) \cdots S^{\bullet}_{d_{r}}(s_{r}) \\
    &= \sum_{i = 1}^{r - 1} \sLbinv(\fs[\,:i]) \sLb(\fs[i + 1:\,]) + (- 1)^{r} \sum_{\substack{0 \leq d_{1} \leq \cdots \leq d_{r} \\ d_{r} \geq 0}} S^{\bullet}_{d_{1}}(s_{1}) \cdots S^{\bullet}_{d_{r}}(s_{r}) \\
    &= \sum_{i = 1}^{r} \sLbinv(\fs[\,:i]) \sLb(\fs[i + 1:\,]).
\end{align*}
\end{proof}

\subsection{Linear relations among non-$*$-inverse values}

Let $\bullet \in \{ \zeta, \Li \}$.
By the results of \cite{CCM23}, \cite{IKLNDP24}, \cite{ND21}, and \cite{To18}, we can explicitly construct a graded $R$-linear map $\sU^{\bullet} \colon \fh^{1}_{R} \to \fh^{1}_{R}$ such that, for each $P \in \fh^{1}_{R}$,
\begin{itemize}
\item $\sL^{\bullet}(\sU^{\bullet}(P)) = \sL^{\bullet}(P)$,
\item there exists $e \in \ZZ_{\geq 0}$ such that $\Supp((\sU^{\bullet})^{e}(P)) \subset \IT$,
\item $\sU^{\bullet}(P) = P$ if $\Supp(P) \subset \IT$.
\end{itemize}
We set
\begin{align*}
    \sR^{\bullet}_{R, w} \coloneq \Span_{R} \{ \fs - \sU^{\bullet}(\fs) \, | \, \fs \in \cI_{w} \setminus \ITw \}
    = \Span_{R} \{ P - \sU^{\bullet}(P) \, | \, P \in \fh^{1}_{R, w} \} \subset \fh^{1}_{R, w}
\end{align*}
for each $w \in \ZZ_{\geq 0}$, and set
\begin{align*}
    \sR^{\bullet}_{R} \coloneq \bigoplus_{w \geq 0} \sR^{\bullet}_{R, w} = \Span_{R} \{ P - \sU^{\bullet}(P) \, | \, P \in \fh^{1}_{R} \} \subset \fh^{1}_{R}.
\end{align*}

\begin{proposition} \label{proposition-generators-relations}
For each $\bullet \in \{ \zeta, \Li \}$, we have
\begin{align*}
    \Ker (\sL^{\bullet}|_{\fh^{1}_{R, w}} \colon \fh^{1}_{R, w} \to \cZ_{R}) = \sR^{\bullet}_{R, w} \ (w \geq 0) \ \ \textrm{and} \ \ \Ker (\sL^{\bullet} \colon \fh^{1}_{R} \to \cZ_{R}) = \sR^{\bullet}_{R}.
\end{align*}
\end{proposition}

\begin{proof}
The case $(R, \bullet) = (k, \zeta)$ was proved in \cite[Theorem~3.6]{CCM23}.
The general case can be proved in a similar manner, but we note that, in general, $R$ is not a field, and the proof requires slight modifications.
Since
\begin{align*}
    \fh^{1}_{R, w} \cap \Ker (\sL^{\bullet} \colon \fh^{1}_{R} \to \cZ_{R}) = \Ker (\sL^{\bullet}|_{\fh^{1}_{R, w}} \colon \fh^{1}_{R, w} \to \cZ_{R}) \ \ \textrm{and} \ \ \fh^{1}_{R, w} \cap \sR^{\bullet}_{R} = \sR^{\bullet}_{R, w}
\end{align*}
for each $w \in \ZZ_{\geq 0}$, it is enough to show the second assertion of the proposition.

The inclusion $\Ker \sL^{\bullet} \supset \sR^{\bullet}_{R}$ follows from the first property of $\sU^{\bullet}$.
Let $P \in \Ker \sL^{\bullet}$.
Let $e \geq 0$ be an integer such that $\Supp((\sU^{\bullet})^{e}(P)) \subset \IT$.
We write $(\sU^{\bullet})^{e}(P) = \sum_{\fs \in \IT} a_{\fs} \fs$ ($a_{\fs} \in R$).
Then we have
\begin{align*}
    \sum_{\fs \in \IT} a_{\fs} \sL^{\bullet}(\fs) = \sL^{\bullet}((\sU^{\bullet})^{e}(P)) = \sL^{\bullet}((\sU^{\bullet})^{e - 1}(P)) = \cdots = \sL^{\bullet}(P) = 0.
\end{align*}
By the linear independence of $\sL^{\bullet}(\fs)$ ($\fs \in \IT$), we have $a_{\fs} = 0$ for all $\fs \in \IT$.
Thus
\begin{align*}
    P \equiv \sU^{\bullet}(P) \equiv (\sU^{\bullet})^{2}(P) \equiv \cdots \equiv (\sU^{\bullet})^{e}(P) = 0 \bmod \sR^{\bullet}_{R},
\end{align*}
and hence $P \in \sR^{\bullet}_{R}$.
Therefore, we have the second assertion of the proposition.
\end{proof}

For each $\bullet \in \{ \zeta, \Li \}$ and $\fc = (c_{1}, \ldots, c_{m}) \in \cI$, let $\alpha^{\bullet}_{\fc} \colon \fh^{1}_{R} \to \fh^{1}_{R}$ be the map defined recursively by
\begin{align*}
    \alpha^{\bullet}_{\fc}(P) \coloneq \left\{ \begin{array}{@{}ll}
    P & (m = 0) \\
    \alpha^{\bullet}_{(c_{1}, \ldots, c_{m - 1})}(c_{m}, (q - 1) *^{\bullet} P) & (m \geq 1) \end{array} \right.
\end{align*}
for $P \in \fh^{1}_{R}$.
By definition, we have
\begin{align*}
    \alpha^{\bullet}_{\fc}(\fn) \in \fh^{1}_{R, (q - 1) \dep(\fc) + \wt(\fc) + \wt(\fn)}
\end{align*}
for each $\fc, \fn \in \cI$.
Finally, let $\boxplus \colon (\fh^{1}_{R})^{2} \to \fh^{1}_{R}$ be the $R$-bilinear map such that
\begin{align*}
    \emptyset \boxplus P = P \boxplus \emptyset = 0 \ \ \textrm{and} \ \ \fs \boxplus \fn = (\fs_{+}, s_{r} + n_{1}, \fn_{-})
\end{align*}
for all $P \in \fh^{1}_{R}$ and $\fs = (s_{1}, \ldots, s_{r}), \fn = (n_{1}, \ldots, n_{\ell}) \in \cI \setminus \{ \emptyset \}$.

For each $\fs = (s_{1}, \ldots, s_{r}) \in \cI \setminus \{ \emptyset \}$, $\fn \in \cI$, and $m \in \ZZ_{\geq 1}$, we define
\begin{align*}
    \sA^{\zeta}(\fs; m; \fn) &\coloneq (\fs, \{ q \}^{m}, \fn) + (\fs, \{ q \}^{m} \boxplus \fn) + (\fs, \{ q \}^{m - 1}, D_{q}(\fn)) - L_{1}^{m} (\fs, \alpha^{\zeta}_{\{ 1 \}^{m}}(\fn)) \\
    &\phantom{\coloneq} \; - L_{1}^{m} (\fs \boxplus \alpha^{\zeta}_{\{ 1 \}^{m}}(\fn)) - L_{1}^{m} (\fs_{+}, D_{s_{r}}(\alpha^{\zeta}_{\{ 1 \}^{m}}(\fn))) \in \fh^{1}_{R}
\end{align*}
and
\begin{align*}
    \sA^{\zeta}(\emptyset; m; \fn) &\coloneq (\{ q \}^{m}, \fn) + (\{ q \}^{m} \boxplus \fn) + (\{ q \}^{m - 1}, D_{q}(\fn)) - L_{1}^{m} (\alpha^{\zeta}_{\{ 1 \}^{m}}(\fn)) \in \fh^{1}_{R}.
\end{align*}
We note that $(\fs_{+}, D_{s_{r}}(\alpha^{\zeta}_{\{ 1 \}^{m}}(\fn))) = 0$ when $s_{r} < q$, in particular for the case $\fs \in \IT$ (see \cite[Remarks~2.3 and 2.6]{CCM23}).
Similarly, for each $\fs, \fn \in \cI$ and $m \in \ZZ_{\geq 1}$, we define
\begin{align*}
    \sA^{\Li}(\fs; m; \fn) &\coloneq (\fs, \{ q \}^{m}, \fn) + (\fs, \{ q \}^{m} \boxplus \fn) \\
    &\phantom{\coloneq} \; - L_{1}^{m} (\fs, \alpha^{\Li}_{\{ 1 \}^{m}}(\fn)) - L_{1}^{m} (\fs \boxplus \alpha^{\Li}_{\{ 1 \}^{m}}(\fn)) \in \fh^{1}_{R}.
\end{align*}

\begin{proposition} \label{proposition-A-span}
We have $\sL^{\bullet}(\sA^{\bullet}(\fs; m; \fn)) = 0$ for all $\bullet \in \{ \zeta, \Li \}$, $\fs, \fn \in \cI$, and $m \in \ZZ_{\geq 1}$.
Moreover, for each $\bullet \in \{ \zeta, \Li \}$, the family
\begin{align*}
    \{ \sA^{\bullet}(\fs; m; \fn) \, | \, \fs, \fn \in \cI, \, m \in \ZZ_{\geq 1} \}
\end{align*}
spans $\Ker (\sL^{\bullet} \colon \fh^{1}_{R} \to \cZ_{R})$ over $R$.
\end{proposition}

\begin{proof}
Fix $\bullet \in \{ \zeta, \Li \}$.
For each $\fs, \fn \in \cI \setminus \{ \emptyset \}$ and $\ell \in \ZZ_{\geq 0}$, let $\sB^{\bullet}_{\fs}, \sC^{\bullet}_{\fn}, \sBC^{\bullet, \ell}_{q}$ be the endomorphisms on $(\bigoplus_{w > 0} \fh^{1}_{R, w})^{2}$ given explicitly in \cite[Section~A.3]{CCM23}.
We note that these maps were defined only in the case of $R = k$, but the definition is the same for general $R$.
We also define $\sB^{\bullet}_{\emptyset}$ and $\sC^{\bullet}_{\emptyset}$ to be the identity map on $(\bigoplus_{w > 0} \fh^{1}_{R, w})^{2}$.
Let $\beta \colon (\bigoplus_{w > 0} \fh^{1}_{R, w})^{2} \to \bigoplus_{w > 0} \fh^{1}_{R, w}$ be the map defined by $\beta(P, Q) \coloneq P + Q$.
Then for each $\fs, \fn \in \cI$ and $m \in \ZZ_{\geq 1}$, we have
\begin{align*}
    \sA^{\bullet}(\fs; m; \fn) = \beta(\sB^{\bullet}_{\fs}(\sBC^{\bullet, m - 1}_{q}(\sC^{\bullet}_{\fn}(R_{1})))),
\end{align*}
where
\begin{align*}
    R_{1} \coloneq (q, - L_{1} (1, q - 1)) \in (\fh^{1}_{R, q})^{2}
\end{align*}
is the fundamental relation corresponding to the relation
\begin{align*}
    S_{d}(q) - L_{1} S_{d + 1}(1) \sum_{i = 0}^{d} S_{i}(q - 1) = 0 \ \ (d \geq 0)
\end{align*}
(which is equal to $\displaystyle \dfrac{1}{L_{d}^{q}} - \dfrac{L_{1}}{L_{d + 1}} \sum_{i = 0}^{d} \dfrac{1}{L_{i}^{q - 1}} = 0$) given by Thakur in \cite[Section~3.4.6]{T09b}.
By \cite[Proposition~A.4]{CCM23}, we have $\sA^{\bullet}(\fs; m; \fn) \in \Ker \sL^{\bullet}$.
Moreover, by the proof of \cite[Theorem~A.5]{CCM23}, for each $\fa \in \cI \setminus \IT$, we can find $\fs, \fn \in \cI$ and $m \in \ZZ_{\geq 1}$ such that
\begin{align*}
    \fa - \sU^{\bullet}(\fa) = \sA^{\bullet}(\fs; m; \fn).
\end{align*}
Since $\{ \fa - \sU^{\bullet}(\fa) \, | \, \fa \in \cI \setminus \IT \}$ spans $\Ker \sL^{\bullet}$ by Proposition~\ref{proposition-generators-relations}, we have the desired result.
\end{proof}

\begin{remark}
Let $\cI' \coloneq \{ \emptyset \} \cup \{ (n_{1}, \ldots, n_{\ell}) \in \cI \setminus \{ \emptyset \} \, | \, n_{1} > q \}$.
Then any $\fa \in \cI$ can be expressed uniquely as $\fa = (\fs, \{ q \}^{m - 1}, \fn)$ with $\fs \in \IT$, $m \in \ZZ_{\geq 1}$, and $\fn \in \cI'$.
In this case, we have $\fa - \sU^{\bullet}(\fa) = \sA^{\bullet}(\fs; m; \fn)$.
\end{remark}

\section{Hopf algebra and comodule structures} \label{section-proof}

\subsection{Key lemma} \label{subsection-key-lemma}

The following lemma is a key ingredient in the proof of Theorem~\ref{theorem-main}.

\begin{lemma} \label{lemma-key}
For each $\fs = (s_{1}, \ldots, s_{r}), \fc = (c_{1}, \ldots, c_{m}), \fn = (n_{1}, \ldots, n_{\ell}) \in \cI$, we have
\begin{align*}
    (\cL \otimes \overline{\cL})(\Delta(\fs, \alpha^{\Li}_{\fc}(\fn)))
    &= \sum_{0 \leq i \leq r} \cL(\fs[\,:i]) \otimes \overline{\cL}(\fs[i + 1:\,], \alpha^{\Li}_{\fc}(\fn)) \\
    &\phantom{=} \; + \sum_{1 \leq i \leq m} \cL(\fs, \alpha^{\Li}_{\fc[\,:i]}(\emptyset)) \otimes \overline{\cL}(\alpha^{\Li}_{\fc[i + 1:\,]}(\fn)) \\
    &\phantom{=} \; + \sum_{1 \leq i \leq \ell} \cL(\fs, \alpha^{\Li}_{\fc}(\fn[\,:i])) \otimes \overline{\cL}(\fn[i + 1:\,])
\end{align*}
in $\cZ_{R} \otimes_{R} \overline{\cZ}_{R}$.
\end{lemma}

\begin{proof}
We prove the lemma by induction on $m = \dep(\fc)$.
When $m = 0$, the equation holds clearly by the definition of $\Delta$.
Let $m \geq 1$ and assume that the equation holds for $\fc_{+} \coloneq \fc[\,:m - 1]$.
We set
\begin{align*}
    \fn_{0, j} \coloneq (\fn[\,:j - 1], q - 1 + n_{j}, \fn[j + 1:\,]) \ \ \ (1 \leq j \leq \ell)
\end{align*}
and
\begin{align*}
    \fn_{1, j} \coloneq (\fn[\,:j - 1], q - 1, \fn[j:\,]) \ \ \ (1 \leq j \leq \ell + 1).
\end{align*}
Then
\begin{align*}
    &\sum_{0 \leq i \leq \ell} ((q - 1) * \fn[\,:i]) \otimes \fn[i + 1:\,] \\
    &= \sum_{1 \leq i \leq \ell} \, \sum_{1 \leq j \leq i} \fn_{0, j}[\,:i] \otimes \fn_{0, j}[i + 1:\,] + \sum_{0 \leq i \leq \ell} \, \sum_{1 \leq j \leq i + 1} \fn_{1, j}[\,:i + 1] \otimes \fn_{1, j}[i + 2:\,] \\
    &= \sum_{0 \leq e \leq 1} \, \sum_{1 \leq j \leq i \leq \ell + e} \fn_{e, j}[\,:i] \otimes \fn_{e, j}[i + 1:\,] \\
\end{align*}
and
\begin{align*}
    &(q - 1) * \fn[i + 1:\,] = \sum_{0 \leq e \leq 1} \, \sum_{i < j \leq \ell + e} \fn_{e, j}[i + 1:\,] \ \ \ (0 \leq i \leq \ell).
\end{align*}
Therefore, by the induction hypothesis, we have
\begin{align*}
    &(\cL \otimes \overline{\cL})(\Delta(\fs, \alpha^{\Li}_{\fc}(\fn))) \\
    &= \sum_{0 \leq e \leq 1} \, \sum_{1 \leq j \leq \ell + e} (\cL \otimes \overline{\cL})(\Delta(\fs, \alpha^{\Li}_{\fc_{+}}(c_{m}, \fn_{e, j}))) \\
    &= \sum_{0 \leq e \leq 1} \, \sum_{1 \leq j \leq \ell + e} \biggl( \sum_{0 \leq i \leq r} \cL(\fs[\,:i]) \otimes \overline{\cL}(\fs[i + 1:\,], \alpha^{\Li}_{\fc_{+}}(c_{m}, \fn_{e, j})) \\
    &\phantom{=} \; + \sum_{1 \leq i \leq m - 1} \cL(\fs, \alpha^{\Li}_{\fc[\,:i]}(\emptyset)) \otimes \overline{\cL}(\alpha^{\Li}_{\fc_{+}[i + 1:\,]}(c_{m}, \fn_{e, j})) \\
    &\phantom{=} \; + \sum_{0 \leq i \leq \ell + e} \cL(\fs, \alpha^{\Li}_{\fc_{+}}(c_{m}, \fn_{e, j}[\,:i])) \otimes \overline{\cL}(\fn_{e, j}[i + 1:\,]) \biggr) \\
    &= \sum_{0 \leq i \leq r} \cL(\fs[\,:i]) \otimes \overline{\cL}(\fs[i + 1:\,], \alpha^{\Li}_{\fc}(\fn)) + \sum_{1 \leq i \leq m - 1} \cL(\fs, \alpha^{\Li}_{\fc[\,:i]}(\emptyset)) \otimes \overline{\cL}(\alpha^{\Li}_{\fc[i + 1:\,]}(\fn)) \\
    &\phantom{=} \; + \sum_{0 \leq e \leq 1} \, \sum_{1 \leq j \leq i \leq \ell + e} \cL(\fs, \alpha^{\Li}_{\fc_{+}}(c_{m}, \fn_{e, j}[\,:i])) \otimes \overline{\cL}(\fn_{e, j}[i + 1:\,]) \\
    &\phantom{=} \; + \sum_{0 \leq e \leq 1} \, \sum_{0 \leq i < j \leq \ell + e} \cL(\fs, \alpha^{\Li}_{\fc_{+}}(c_{m}, \fn[\,:i])) \otimes \overline{\cL}(\fn_{e, j}[i + 1:\,]) \\
    &= \sum_{0 \leq i \leq r} \cL(\fs[\,:i]) \otimes \overline{\cL}(\fs[i + 1:\,], \alpha^{\Li}_{\fc}(\fn)) + \sum_{1 \leq i \leq m - 1} \cL(\fs, \alpha^{\Li}_{\fc[\,:i]}(\emptyset)) \otimes \overline{\cL}(\alpha^{\Li}_{\fc[i + 1:\,]}(\fn)) \\
    &\phantom{=} \; + \sum_{0 \leq i \leq \ell} \cL(\fs, \alpha^{\Li}_{\fc}(\fn[\,:i])) \otimes \overline{\cL}(\fn[i + 1:\,]) \\
    &\phantom{=} \; + \sum_{0 \leq i \leq \ell} \cL(\fs, \alpha^{\Li}_{\fc_{+}}(c_{m}, \fn[\,:i])) \otimes \overline{\cL}((q - 1) * \fn[i + 1:\,]).
\end{align*}
Since, by \eqref{equation-product-formula}, $\overline{\cL}((q - 1) * \fn[i + 1:\,]) = \overline{\cL}(q - 1) \overline{\cL}(\fn[i + 1:\,]) = 0$, we obtain the desired result.
\end{proof}

\subsection{Proof of the main theorem}

In this subsection, we provide the proof of Theorem~\ref{theorem-main}.
We note that $\cL$ and $\ocL$ are $R$-algebra homomorphisms by \cite[(5.2.1)]{C14} (see also \eqref{equation-product-formula}), and that
\begin{align*}
    (\ocL \otimes \ocL)(\Delta(q - 1)) = \ocL(\emptyset) \otimes \ocL(q - 1) + \ocL(q - 1) \otimes \ocL(\emptyset) = 0
\end{align*}
in $\ocZ_{R} \otimes_{R} \ocZ_{R}$.
Therefore, by Proposition~\ref{proposition-A-span}, to prove the existence of the $R$-algebra homomorphisms $\Delta_{\cZ} \colon \cZ_{R} \to \cZ_{R} \otimes_{R} \ocZ_{R}$ and $\overline{\Delta} \colon \ocZ_{R} \to \ocZ_{R} \otimes_{R} \ocZ_{R}$, it suffices to show that the equality
\begin{align} \label{equation-Delta-existence}
    (\cL \otimes \overline{\cL})(\Delta(\sA^{\Li}(\fs; m; \fn))) = 0
\end{align}
holds in $\cZ_{R} \otimes_{R} \overline{\cZ}_{R}$ for all $\fs, \fn \in \cI$ and all $m \in \ZZ_{\geq 1}$.

Let $r \coloneq \dep(\fs)$ and $\ell \coloneq \dep(\fn)$.
We note that
\begin{align*}
    \fs \boxplus \alpha^{\Li}_{\{ 1 \}^{m}}(\fn) = (\fs[\,:r - 1], \alpha^{\Li}_{(s_{r} + 1, \{ 1 \}^{m - 1})}(\fn))
\end{align*}
if $\fs = (s_{1}, \ldots, s_{r}) \neq \emptyset$.
Using Lemma~\ref{lemma-key} and the equality $\fs[r + 1:\,] \boxplus \alpha^{\Li}_{\{ 1 \}^{m}}(\fn) = 0$, we obtain
\begin{align*}
    &(\cL \otimes \overline{\cL})(\Delta(\sA^{\Li}(\fs; m; \fn))) \\
    &= (\cL \otimes \overline{\cL})(\Delta(\fs, \{ q \}^{m}, \fn)) + (\cL \otimes \overline{\cL})(\Delta(\fs, \{ q \}^{m} \boxplus \fn)) \\
    &\phantom{=} \; - L_{1}^{m} (\cL \otimes \overline{\cL})(\Delta(\fs, \alpha^{\Li}_{\{ 1 \}^{m}}(\fn))) - L_{1}^{m} (\cL \otimes \overline{\cL})(\Delta(\fs \boxplus \alpha^{\Li}_{\{ 1 \}^{m}}(\fn))) \\
    &= \sum_{0 \leq i \leq r} \cL(\fs[\,:i]) \otimes \overline{\cL}(\fs[i + 1:\,], \{ q \}^{m}, \fn) + \sum_{1 \leq i \leq m} \cL(\fs, \{ q \}^{i}) \otimes \overline{\cL}(\{ q \}^{m - i}, \fn) \\
    &\phantom{=} \; + \sum_{1 \leq i \leq \ell} \cL(\fs, \{ q \}^{m}, \fn[\,:i]) \otimes \overline{\cL}(\fn[i + 1:\,]) \\
    &\phantom{=} \; + \sum_{0 \leq i \leq r} \cL(\fs[\,:i]) \otimes \overline{\cL}(\fs[i + 1:\,], \{ q \}^{m} \boxplus \fn) + \sum_{1 \leq i \leq m - 1} \cL(\fs, \{ q \}^{i}) \otimes \overline{\cL}(\{ q \}^{m - i} \boxplus \fn) \\
    &\phantom{=} \; + \sum_{1 \leq i \leq \ell} \cL(\fs, \{ q \}^{m} \boxplus \fn[\,:i]) \otimes \overline{\cL}(\fn[i + 1:\,]) \\
    &\phantom{=} \; - L_{1}^{m} \sum_{0 \leq i \leq r} \cL(\fs[\,:i]) \otimes \overline{\cL}(\fs[i + 1:\,], \alpha^{\Li}_{\{ 1 \}^{m}}(\fn)) \\
    &\phantom{=} \; - L_{1}^{m} \sum_{1 \leq i \leq m} \cL(\fs, \alpha^{\Li}_{\{ 1 \}^{i}}(\emptyset)) \otimes \overline{\cL}(\alpha^{\Li}_{\{ 1 \}^{m - i}}(\fn)) \\
    &\phantom{=} \; - L_{1}^{m} \sum_{1 \leq i \leq \ell} \cL(\fs, \alpha^{\Li}_{\{ 1 \}^{m}}(\fn[\,:i])) \otimes \overline{\cL}(\fn[i + 1:\,]) \\
    &\phantom{=} \; - L_{1}^{m} \sum_{0 \leq i \leq r} \cL(\fs[\,:i]) \otimes \overline{\cL}(\fs[i + 1:\,] \boxplus \alpha^{\Li}_{\{ 1 \}^{m}}(\fn)) \\
    &\phantom{=} \; - L_{1}^{m} \sum_{1 \leq i \leq m} \cL(\fs \boxplus \alpha^{\Li}_{\{ 1 \}^{i}}(\emptyset)) \otimes \overline{\cL}(\alpha^{\Li}_{\{ 1 \}^{m - i}}(\fn)) \\
    &\phantom{=} \; - L_{1}^{m} \sum_{1 \leq i \leq \ell} \cL(\fs \boxplus \alpha^{\Li}_{\{ 1 \}^{m}}(\fn[\,:i])) \otimes \overline{\cL}(\fn[i + 1:\,]) \\
    &= \sum_{0 \leq i \leq r} \cL(\fs[\,:i]) \otimes \overline{\cL}(\sA^{\Li}(\fs[i + 1:\,]; m; \fn)) + \sum_{1 \leq i \leq m - 1} \cL(\fs, \{ q \}^{i}) \otimes \overline{\cL}(\sA^{\Li}(\emptyset; m - i; \fn)) \\
    &\phantom{=} \; + \sum_{1 \leq i \leq \ell} \cL(\sA^{\Li}(\fs; m; \fn[\,:i])) \otimes \overline{\cL}(\fn[i + 1:\,]).
\end{align*}
Here, we used
\begin{align*}
    \cL(\fs, \{ q \}^{i}) - L_{1}^{i} \cL(\fs, \alpha^{\Li}_{\{ 1 \}^{i}}(\emptyset)) - L_{1}^{i} \cL(\fs \boxplus \alpha^{\Li}_{\{ 1 \}^{i}}(\emptyset)) = \cL(\sA^{\Li}(\fs; i; \emptyset)) = 0
\end{align*}
for each $1 \leq i \leq m$.
Therefore, \eqref{equation-Delta-existence} follows from Proposition~\ref{proposition-A-span}, and hence $\Delta_{\cZ}$ and $\overline{\Delta}$ are well-defined.

Since the Hopf algebra $(\fh^{1}_{R}, *, \eta, \Delta, \varepsilon, S)$ is connected and graded, and the $R$-algebra homomorphism $\ocL \colon \fh^{1}_{R} \twoheadrightarrow \ocZ_{R}$ preserves the grading, the standard argument for quotients of connected graded Hopf algebras shows that the maps $\varepsilon$ and $S$ descend to $\ocZ_{R}$, endowing $\ocZ_{R}$ with a connected graded Hopf algebra structure.
Since $S$ is involutive, so is $\overline{S}$.
Moreover, the coassociativity of $\Delta$ and the counit property of $\varepsilon$ imply that $\cZ_{R}$ is a graded right comodule over $\ocZ_{R}$.

Finally, we show that the involution $\overline{S} \colon \ocZ_{R} \to \ocZ_{R}$ is non-trivial.
Recall the decomposition
\begin{align*}
    \ocZ_{R} = \biggl( \bigoplus_{0 \leq w \leq q - 2} \cZ_{R, w} \biggr) \oplus \biggl( \bigoplus_{w \geq q - 1} \cZ_{R, w} / \zeta_{A}(q - 1) \cZ_{R, w - (q - 1)} \biggr).
\end{align*}
If $p \neq 2$, then
\begin{align*}
    \cLinv(1) - \cL(1) = - 2 \cL(1) \not\equiv 0 \bmod \zeta_{A}(q - 1) \cZ_{R}.
\end{align*}
Similarly, if $q \geq 4$, then
\begin{align*}
    \cLinv(1, 1) - \cL(1, 1) = \cL(2) \not\equiv 0 \bmod \zeta_{A}(q - 1) \cZ_{R}.
\end{align*}
Let $q = 2$.
According to Proposition~\ref{proposition-A-span} and the harmonic product $(1) * \fs$ ($\fs \in \cI_{5}$), the subspace $\cZ_{\ok, 6} / \zeta_{A}(1) \cZ_{\ok, 5} \subset \ocZ_{\ok}$ is a three-dimensional $\ok$-vector space with basis
\begin{align*}
    \cL(6) \bmod \zeta_{A}(1) \cZ_{\ok}, \ \ \cL(5, 1) \bmod \zeta_{A}(1) \cZ_{\ok}, \ \ \cL(3, 3) \bmod \zeta_{A}(1) \cZ_{\ok}.
\end{align*}
In particular, $\cL(6) \notin \zeta_{A}(1) \cZ_{\ok}$.
Since there is a natural map $\cZ_{R} / \zeta_{A}(1) \cZ_{R} \to \cZ_{\ok} / \zeta_{A}(1) \cZ_{\ok}$, we have
\begin{align*}
    \cLinv(3, 3) - \cL(3, 3) = \cL(6) \not\equiv 0 \bmod \zeta_{A}(1) \cZ_{R}.
\end{align*}
This completes the proof of Theorem~\ref{theorem-main}.

\section{Relations among $*$-inverse MZVs} \label{section-*-inverse-MZV}

According to Conjecture~\ref{conjecture-zeta-image}, we expect that, if $\Fp(L_{1}) \subset R$, the values $\zetainv_{A}(\fs) \bmod \zeta_{A}(q - 1) \cZ_{R}$ satisfy the $q$-shuffle product formula and the linear relations associated with $\sA^{\zeta}(\fs; m; \fn)$.
In this section, we establish these relations in some special cases.
We also consider the corresponding question for multiple zeta values in characteristic zero, where the naive analogue may fail, and discuss a possible modification.

\subsection{$q$-shuffle product formula among $*$-inverse single zeta values}

So far, the $q$-shuffle product formula for $*$-inverse MZVs has been obtained only for the depth one case.
We note that, unlike in the case of $\cLinv(\fs)$ in Proposition~\ref{proposition-Li-*-inv-prod}, the $q$-shuffle product formula holds only after taking the quotient.

\begin{proposition} \label{proposition-zeta-*-inv-prod}
Assume that $\Fp(L_{1}) \subset R$.
Then we have
\begin{align*}
    \zetainv_{A}(s) \zetainv_{A}(n) \equiv \zetainv_{A}((s) *^{\zeta} (n)) \bmod \zeta_{A}(q - 1) \cZ_{R}
\end{align*}
for all $s, n \in \ZZ_{\geq 1}$.
\end{proposition}

\begin{proof}
We have
\begin{align*}
    \zetainv_{A}(s) \zetainv_{A}(n)
    &= \sum_{0 \leq d \leq e} S_{d}(s) S_{e}(n) + \sum_{0 \leq e \leq d} S_{d}(s) S_{e}(n) - \sum_{0 \leq d} S_{d}(s) S_{d}(n) \\
    &= \zetainv_{A}(s, n) + \zetainv_{A}(n, s) - \zeta_{A}(s + n) - \sum_{j = 1}^{s + n - 1} \Delta_{s, n}^{[j]} \zeta_{A}(s + n - j, j) \\
    &= \zetainv_{A}(s, n) + \zetainv_{A}(n, s) + \zetainv_{A}(s + n) \\
    &\phantom{=} \; + \sum_{j = 1}^{s + n - 1} \Delta_{s, n}^{[j]} \left( \zetainv_{A}(s + n - j, j) + \zeta_{A}(s + n - j) \zetainv_{A}(j) \right) \\
    &= \zetainv_{A}((s) *^{\zeta} (n)) + \sum_{j = 1}^{s + n - 1} \Delta_{s, n}^{[j]} \zeta_{A}(s + n - j) \zetainv_{A}(j),
\end{align*}
where the second equality follows from \cite[Remark~3.2]{Ch15} and the third equality follows from Proposition~\ref{proposition-prod-sum}.
By definition, $\Delta_{s, n}^{[j]} = 0$ when $(q - 1) \nmid j$, whereas for $(q - 1) \mid j$, we have
\begin{align*}
    \zetainv_{A}(j) = - \zeta_{A}(j) \in \Fp(L_{1})^{\times} \cdot \zeta_{A}(q - 1)^{j / (q - 1)} \subset R \cdot \zeta_{A}(q - 1)^{j / (q - 1)}
\end{align*}
by \eqref{equation-Euler-Carlitz}.
Therefore, we obtain the $q$-shuffle product formula in this case.
\end{proof}

\subsection{Linear relations among $*$-inverse MZVs}

So far, the linear relations associated with $\sA^{\zeta}(\fs; m; \fn)$ for $*$-inverse MZVs have been obtained only in the following cases:
\begin{proposition} \label{proposition-zeta-A}
Assume that $\Fp(L_{1}) \subset R$.
Let $\fs, \fn \in \cI$ satisfy the following conditions:
\begin{itemize}
\item $\fs = \emptyset$ or $\fs = (s_{1}, \ldots, s_{r}) \neq \emptyset$ with $s_{r} < q$,
\item $\dep(\fn) \leq 1$.
\end{itemize}
Then we have
\begin{align*}
    \zetainv_{A}(\sA^{\zeta}(\fs; 1; \fn)) \equiv 0 \bmod \zeta_{A}(q - 1) \cZ_{R}.
\end{align*}
\end{proposition}

\begin{proof}
Let $r \coloneq \dep(\fs)$.
By the assumption on $\fs$,
\begin{align*}
    \sA^{\zeta}(\fs[i + 1:\,]; 1; \fn) &= (\fs[i + 1:\,], q, \fn) + (\fs[i + 1:\,], \{ q \} \boxplus \fn) + (\fs[i + 1:\,], D_{q}(\fn)) \\
    &\phantom{=} \; - L_{1} (\fs[i + 1:\,], 1, (q - 1) *^{\zeta} \fn) - L_{1} (\fs[i + 1:\,] \boxplus (1), (q - 1) *^{\zeta} \fn)
\end{align*}
for each $0 \leq i \leq r$.
If $\fn = (n)$ with $n \geq 1$, then by \eqref{equation-Euler-Carlitz} and Propositions~\ref{proposition-prod-sum}, \ref{proposition-A-span}, and \ref{proposition-zeta-*-inv-prod}, we have
\begin{alignat*}{2}
    &\zetainv_{A}(\sA^{\zeta}(\fs; 1; n)) \\
    &= - \sum_{1 \leq i \leq r} \zeta_{A}(\fs[\,:i]) \zetainv_{A}(\fs[i + 1:\,], q, n) - \zeta_{A}(\fs, q) \zetainv_{A}(n) - \zeta_{A}(\fs, q, n) \\
    &\phantom{=} \; - \sum_{1 \leq i \leq r} \zeta_{A}(\fs[\,:i]) \zetainv_{A}(\fs[i + 1:\,], q + n) - \zeta_{A}(\fs, q + n) \\
    &\phantom{=} \; - \sum_{1 \leq i \leq r} \zeta_{A}(\fs[\,:i]) \zetainv_{A}(\fs[i + 1:\,], D_{q}(n)) \\
    &\phantom{=} \; - \sum_{1 \leq j < q + n} \Delta_{q, n}^{[j]} \zeta_{A}(\fs, q + n - j) \zetainv_{A}(j) - \zeta_{A}(\fs, D_{q}(n)) \\
    &\phantom{=} \; + L_{1} \sum_{1 \leq i \leq r} \zeta_{A}(\fs[\,:i]) \zetainv_{A}(\fs[i + 1:\,], 1, (q - 1) *^{\zeta} (n)) + L_{1} \zeta_{A}(\fs, 1) \zetainv_{A}((q - 1) *^{\zeta} (n)) \\
    &\phantom{=} \; + L_{1} \zeta_{A}(\fs, 1, (q - 1) *^{\zeta} (n)) + L_{1} \zeta_{A}(\fs, 1, q - 1) \zetainv_{A}(n) + L_{1} \zeta_{A}(\fs, 1, n) \zetainv_{A}(q - 1) \\
    &\phantom{=} \; + L_{1} \sum_{1 \leq j < q - 1 + n} \Delta_{q - 1, n}^{[j]} \zeta_{A}(\fs, 1, q - 1 + n - j) \zetainv_{A}(j) \\
    &\phantom{=} \; + L_{1} \sum_{1 \leq i \leq r} \zeta_{A}(\fs[\,:i]) \zetainv_{A}(\fs[i + 1:\,] \boxplus (1), (q - 1) *^{\zeta} (n)) \\
    &\phantom{=} \; + L_{1} \zeta_{A}(\fs \boxplus (1)) \zetainv_{A}((q - 1) *^{\zeta} (n)) + L_{1} \zeta_{A}(\fs \boxplus (1), (q - 1) *^{\zeta} (n)) \\
    &\phantom{=} \; + L_{1} \zeta_{A}(\fs \boxplus (1), q - 1) \zetainv_{A}(n) + L_{1} \zeta_{A}(\fs \boxplus (1), n) \zetainv_{A}(q - 1) \\
    &\phantom{=} \; + L_{1} \sum_{1 \leq j < q - 1 + n} \Delta_{q - 1, n}^{[j]} \zeta_{A}(\fs \boxplus (1), q - 1 + n - j) \zetainv_{A}(j) \\
    &\equiv - \sum_{1 \leq i \leq r} \zeta_{A}(\fs[\,:i]) \zetainv_{A}(\sA^{\zeta}(\fs[i + 1:\,]; 1; n)) - \zeta_{A}(\sA^{\zeta}(\fs; 1; \emptyset)) \zetainv_{A}(n) - \zeta_{A}(\sA^{\zeta}(\fs; 1; n)) \\
    && \mathllap{\bmod \, \zeta_{A}(q - 1) \cZ_{R}} \\
    &= - \sum_{1 \leq i \leq r} \zeta_{A}(\fs[\,:i]) \zetainv_{A}(\sA^{\zeta}(\fs[i + 1:\,]; 1; n)).
\end{alignat*}
Then we obtain $\zetainv_{A}(\sA^{\zeta}(\fs; 1; n)) \equiv 0 \bmod \zeta_{A}(q - 1) \cZ_{R}$ by induction on $r$.
Similarly, if $\fn = \emptyset$, then by Propositions~\ref{proposition-prod-sum} and \ref{proposition-A-span}, together with induction on $r$, we have
\begin{align*}
    \zetainv_{A}(\sA^{\zeta}(\fs; 1; \emptyset))
    &= - \sum_{1 \leq i \leq r} \zeta_{A}(\fs[\,:i]) \zetainv_{A}(\fs[i + 1:\,], q) - \zeta_{A}(\fs, q) \\
    &\phantom{=} \; + L_{1} \sum_{1 \leq i \leq r} \zeta_{A}(\fs[\,:i]) \zetainv_{A}(\fs[i + 1:\,], 1, q - 1) \\
    &\phantom{=} \; + L_{1} \zeta_{A}(\fs, 1) \zetainv_{A}(q - 1) + L_{1} \zeta_{A}(\fs, 1, q - 1) \\
    &\phantom{=} \; + L_{1} \sum_{1 \leq i \leq r} \zeta_{A}(\fs[\,:i]) \zetainv_{A}(\fs[i + 1:\,] \boxplus (1), q - 1) \\
    &\phantom{=} \; + L_{1} \zeta_{A}(\fs \boxplus (1)) \zetainv_{A}(q - 1) + L_{1} \zeta_{A}(\fs \boxplus (1), q - 1) \\
    &\equiv - \sum_{1 \leq i \leq r} \zeta_{A}(\fs[\,:i]) \zetainv_{A}(\sA^{\zeta}(\fs[i + 1:\,]; 1; \emptyset)) - \zeta_{A}(\sA^{\zeta}(\fs; 1; \emptyset)) \\
    &\equiv 0 \bmod \zeta_{A}(q - 1) \cZ_{R}.
\end{align*}
\end{proof}

\begin{corollary} \label{corollary-zeta-U}
Assume that $\Fp(L_{1}) \subset R$.
Then we have
\begin{align*}
    \zetainv_{A}(\sU^{\zeta}(s)) \equiv \zetainv_{A}(s) \bmod \zeta_{A}(q - 1) \cZ_{R}
\end{align*}
for all $s \in \ZZ_{\geq 1}$.
\end{corollary}

\begin{proof}
By the explicit description of $\sU^{\zeta}$ in \cite[Definition~3.4]{CCM23}, we have
\begin{align*}
    (s) - \sU^{\zeta}(s) = \left\{ \begin{array}{@{}ll} 0 & (1 \leq s < q) \\ \sA^{\zeta}(\emptyset; 1; \emptyset) & (s = q) \\ \sA^{\zeta}(\emptyset; 1; s - q) & (s > q) \end{array} \right..
\end{align*}
The result now follows from Proposition~\ref{proposition-zeta-A}.
\end{proof}

\subsection{$*$-inverse MZVs in characteristic zero} \label{subsection-*-inverse-MZV-char-zero}

In this subsection, we discuss the $*$-inverse MZVs in characteristic zero.
Let $\overleftarrow{\cI}^{\adm}$ be the set of indices $\fs = (s_{1}, \ldots, s_{r}) \in \cI$ with $\fs \neq \emptyset$ and $s_{r} > 1$, or $\fs = \emptyset$.

\begin{definition}
For each $\fs = (s_{1}, \ldots, s_{r}) \in \overleftarrow{\cI}^{\adm}$, we define the $*$-inverse multiple zeta value $\zetainv(\fs)$ by
\begin{align*}
    \zetainv(\fs) \coloneq (- 1)^{\dep(\fs)} \sum_{1 \leq m_{1} \leq \cdots \leq m_{r}} \dfrac{1}{m_{1}^{s_{1}} \cdots m_{r}^{s_{r}}} \in \RR.
\end{align*}
\end{definition}

\begin{remark}
When $s_{r} = 1$, there are several choices for the definition of $\zetainv(\fs)$.
Let $\zeta^{*}(\fs)$ (resp.\ $\zeta^{\sh}(\fs)$) be the harmonic (resp.\ shuffle) regularization of MZVs.
For each $\bullet \in \{ *, \sh \}$ and each $\fs = (s_{1}, \ldots, s_{r}) \in \cI$, we define the regularization of the $*$-inverse MZV by
\begin{align*}
    \zeta^{*\textnormal{-inv}, \bullet}(\fs) \coloneq \zeta^{\bullet}(S(\fs)).
\end{align*}
Here, we use the same notation $\zeta^{\bullet}$ for its $\QQ$-linear extension to $\fh^{1}_{\QQ}$, and $S \colon \fh^{1}_{\QQ} \to \fh^{1}_{\QQ}$ is defined by the same formula as in \eqref{equation-antipode}.
This definition is inspired by the regularization $\zeta^{\star, \bullet}(\fs)$ of MZSVs introduced by Muneta \cite[Section~2.3]{Mu09}.
We note that another regularization of MZSVs was introduced by Kaneko and Yamamoto \cite[Section~4]{KY18}, which was defined in terms of the iterated integral expressions of MZSVs studied by Yamamoto \cite{Ya17}.
According to Hirose, Murahara, and Ono \cite[Theorem~2.1]{HMO21}, this regularization is congruent to $\zeta^{\star, \sh}(\fs)$ modulo $\zeta(2) \fZ_{\QQ}$.
\end{remark}

The following example suggests that direct analogues of Theorem~\ref{theorem-main} and Conjecture~\ref{conjecture-zeta-image} may fail to hold in characteristic zero.

\begin{example}
By duality, we have $\zeta(2, 4) = \zeta(2, 1, 1, 2)$.
On the other hand,
\begin{align*}
    \zetainv(2, 4) \equiv - \zetainv(2, 1, 1, 2) \equiv \zeta(3)^{2} \bmod \zeta(2) \fZ_{\QQ}.
\end{align*}
Since it is conjectured that $\zeta(3)^{2} \not\equiv 0 \bmod \zeta(2) \fZ_{\QQ}$, we may have
\begin{align*}
    \zetainv(2, 4) \not\equiv \zetainv(2, 1, 1, 2) \bmod \zeta(2) \fZ_{\QQ}.
\end{align*}
\end{example}

\begin{remark}
Kaneko and Ohno \cite[Theorem~1.1]{KO10} proved a kind of duality for MZSVs.
In our notation, it is stated as
\begin{align*}
    \zetainv(\{ 1 \}^{a}, b + 1) - \zetainv(\{ 1 \}^{b}, a + 1) \in \QQ[\zeta(2), \zeta(3), \zeta(5), \ldots]
\end{align*}
for each $a, b \in \ZZ_{\geq 1}$.
Note that $(\{ 1 \}^{a}, b + 1) \in \overleftarrow{\cI}^{\adm} \setminus \cI^{\adm}$.
\end{remark}

We note that the derivation relations given by Ihara, Kaneko, and Zagier \cite[Corollary~6]{IKZ06} imply that the values $\zeta(\fs)$ ($\fs \in \cI^{\adm} \cap \overleftarrow{\cI}^{\adm}$) generate $\fZ_{\QQ}$ (as a $\QQ$-vector space).
Therefore, we consider the following problem:

\begin{problem}
(1)
Find a proper ideal $J \subsetneq \fZ_{\QQ}$ such that the assignment
\begin{align*}
    \zeta(\fs) \bmod J \mapsto \zetainv(\fs) \bmod J \ \ \ (\fs \in \cI^{\adm} \cap \overleftarrow{\cI}^{\adm})
\end{align*}
induces a well-defined $\QQ$-algebra involution on $\fZ_{\QQ} / J$.

\noindent
(2)
Determine the smallest ideal $J \subset \fZ_{\QQ}$ as in (1).

\noindent
(3)
In (1) and (2), replace $\fZ_{\QQ}$ by the space of formal multiple zeta values, namely, the quotient of the $\QQ$-vector space with basis $\cI^{\adm}$ by the extended double shuffle relations.
\end{problem}

\section*{Acknowledgments}

The author would like to thank Chieh-Yu Chang and Yen-Tsung Chen for helpful discussions, from which the ideas of this paper emerged.
He also thanks Hidekazu Furusho for informing him of the notion of inversed multiple zeta values and for helpful discussions related to the present work.
In particular, Furusho's question about whether the Hopf algebra structure could be accompanied by a corresponding coaction on the original space led the author to the comodule structure in Theorem~\ref{theorem-main}.
The author is grateful to Hung-Chun Tsui for pointing out the necessity of the assumption $\Fp(L_{1}) \subset R$ in Section~\ref{section-*-inverse-MZV}.

\end{document}